\documentclass{article}

\usepackage{amsmath,amsthm,amssymb}
\usepackage[hidelinks]{hyperref}
\usepackage{xcolor}
\usepackage[noabbrev,capitalise]{cleveref}

\newtheorem{thm}{Theorem}
\newtheorem{cor}{Corollary}
\newtheorem{lem}{Lemma}
\newtheorem{conj}{Conjecture}

\newtheorem{quest}{Question}

\usepackage{enumitem}
\usepackage{graphicx}

\newcommand{\dist}{\operatorname{dist}}
\newcommand{\Int}{\operatorname{Int}}
\newcommand{\Out}{\operatorname{Out}}

\title{Partitioning a Planar Graph into two Triangle-Forests}
\author{Kolja Knauer, Clément Rambaud, and Torsten Ueckerdt}
\date{\today}

\begin{document}

\maketitle

\begin{abstract}
    We show that the vertices of every planar graph can be partitioned into two sets, each inducing a so-called triangle-forest, i.e., a graph with no cycles of length more than three. We further discuss extensions to locally planar graphs. 
    After finishing the paper we noticed that our main result was already proved much earlier by Carsten Thomassen [Decomposing a Planar Graph into Degenerate Graphs, JCTB 1995].
\end{abstract}

\section{Introduction}

Tait's conjecture from 1884~\cite{Tai84} is (equivalent to) the claim that the vertices of any planar graph~$G$ can be partitioned into two sets each inducing a forest, also see~\cite{CK69,RW08}. However, the conjecture is false, due to first examples of Tutte~\cite{Tut46}. Also see~\cite{HM88} for further counterexamples.
However, every planar graph can be vertex-partitioned into two outerplanar graphs, by assigning the layers of a BFS-tree alternatingly to the two parts. 
See~\cite{AEFG16} for further considerations into this direction.
In the present paper, we consider a family of graphs strictly between forests and outerplanar graphs:
We call a graph~$G$ a \emph{triangle-forest} if~$G$ has no cycles of length at least four.
In other words, every maximal~$2$-connected subgraph of~$G$ is a triangle, see \cref{fig:triangle-forest}.
We show that the vertices of any planar graph can be partitioned into two sets each inducing a triangle-forest (cf.~\Cref{thm:bipartition}). However, this was already obtained previously in~\cite[Theorem 4.1]{Tho95}. We then show that this does not extend to locally planar graphs (cf.~\Cref{cor:not2}), while it follows from a result of Kawarabayashi and Mohar~\cite{Kawarabayashi2010} that all such graphs can be partitioned into four (triangle-)forests~(cf.~\Cref{cor:4}). 

\begin{figure}
    \centering
     \includegraphics[width=.4\textwidth]{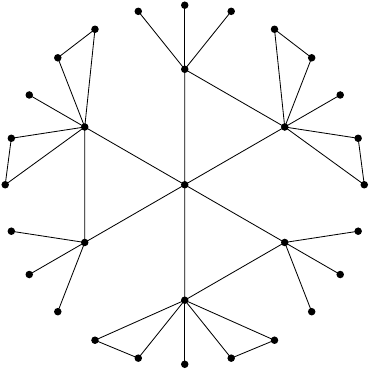}
    \caption{A connected triangle forest.}
    \label{fig:triangle-forest}
\end{figure}




\section{Vertex-Partitioning a Planar Graph into two Triangle-Forests}

We will need the following definitions.
A \emph{Tutte path}, respectively \emph{Tutte cycle}, of a graph~$G$ is a path, respectively cycle,~$T$ such that for any connected component~$K$ of~$G\setminus V(T)$ there are at most~$3$ edges from~$K$ to~$T$. 
Note that originally the definition of Tutte paths has a further stronger property, which we omit here, since we will not need it.
We will use two results on these objects:

\begin{lem}[Tutte~1956~\cite{Tut56}]\label{Tutte}{\ \\}
    Let~$C$ be the outer face of a~$2$-connected planar graph~$G$ and~$u,v,e\in C$ be two vertices and an edge of~$C$.
    Then there exists a Tutte path from~$u$ to~$v$ through~$e$ in~$G$.
\end{lem}

\begin{lem}[Three-Edge-Lemma~\cite{San96,TY94}]\label{3EL}{\ \\}
    Let~$C$ be the outer face of a~$2$-connected planar graph~$G$ and~$e,f,g\in C$ be three edges of~$C$.
    Then there exists a Tutte cycle using edges~$e,f,g$ in~$G$.
\end{lem}

An \emph{edge-cut} of a connected graph~$G$ is a set of edges~$F$, such that~$G\setminus F$ is disconnected.
A \emph{cyclic} edge-cut~$F$ furthermore has the property that each component of~$G\setminus F$ contains a cycle.
The \emph{cyclic} edge-connectivity of~$G$ is the smallest size of a cyclic edge-cut. 
It is easy to see that a~$3$-connected cubic graph is cyclically~$4$-edge-connected if every edge-cut~$F$ of order~$3$ isolates a single vertex.

The following lemma explains why we are interested in the previous lemmas.

\begin{lem}\label{dualtuttecycle}
    If~$T$ is a Tutte cycle in a cyclically~$4$-edge-connected cubic planar graph~$G$, then~$2$-coloring the vertices of~$G^*$ depending on whether they correspond to faces in the interior or exterior of~$T$, yields a partition into two triangle-forests.
\end{lem}
\begin{proof}
    Without loss of generality consider a component~$K$ of~$G^*\setminus V(T)$ in the interior of~$T$.
    Since~$T$ is a Tutte cycle there are at most~$3$ edges from~$K$ to~$T$. 
    Hence, these edges form an edge-cut of size at most~$3$.
    By cyclic connectivity of~$G$, we get that~$K$ is a single vertex.

    Thus, all components of~$G^*\setminus V(T)$ are single vertices, which means that~$G$'s vertices corresponding to faces inside~$T$ do not induce any cycles except faces.
    This means that there are no cycles on more than~$3$ vertices, i.e., they induce a triangle-forest.
\end{proof}

The proof of the following is inspired by a proof in~\cite{LM17} and has been shown in a slightly weaker form in~\cite{Wu10}.

\begin{lem}\label{4connected}
    Let~$G$ be a~$4$-connected planar triangulation with outer triangle~$\Delta$. 
    Then any~$2$-coloring of the vertices of~$\Delta$ extends to a~$2$-coloring of~$G$ such that each color induces a triangle-forest and no edge of~$\Delta$ is in a monochromatic triangle except if~$\Delta$ is monochromatic itself.
\end{lem}
\begin{proof}
    Let~$\Delta=(a,b,c)$ and denote by~$A,B,C$ the interior face of~$G$ that contains the edge~$bc,ac,ab$, respectively.
    Consider the dual graph~$G^*$, which is a~$3$-connected cyclically~$4$-edge-connected, cubic, planar graph.
    Denote by~$\Delta^*$ the vertex of~$G^*$ corresponding to~$\Delta$, by~$A^*,B^*,C^*$ its three neighbors corresponding to the faces~$A,B,C$ of~$G$, and by~$a^*,b^*,c^*$ its three incident faces corresponding to the vertices~$a,b,c$ of~$G$.
    Now, let~$H$ be the graph~$G^*\setminus \Delta^*$, which is~$2$-connected since~$G^*$ is~$3$-connected.
    Moreover, all~$A^*,B^*,C^*$ lie on the outer face~$C$.
    See~\Cref{tuttecycles} for an illustration of~$G^*$ and the following constructions.

    \begin{figure}[htp]
        \centering
        \includegraphics[width=\textwidth]{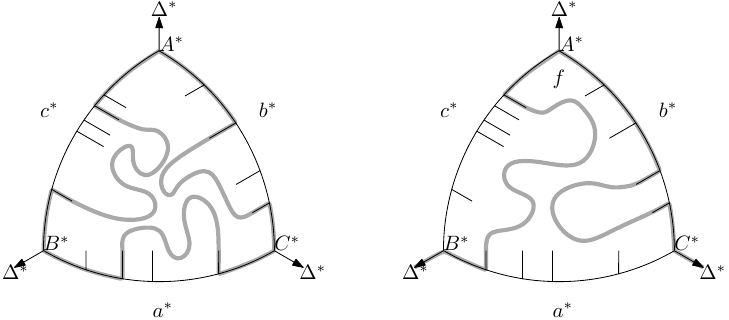}
        \caption{The monochromatic and heterochromatic case of construction of Tutte cycles in the proof of~\Cref{4connected}.}\label{tuttecycles}
    \end{figure}

    In order to construct the desired~$2$-coloring of~$G$, we distinguish the two different possible~$2$-colorings of~$\Delta$. 
    If~$\Delta$ is monochromatic, then we choose three edges on the outer face~$C$ of~$H$ that contain vertices~$A^*,B^*,C^*$.
    By \cref{3EL}, we get a Tutte cycle~$T$ of~$H$ containing~$A^*,B^*,C^*$. Since~$\Delta^*$ has all its neighbors in~$T$, also in~$G^*$ we have that~$T$ is a Tutte cycle.
    Furthermore, observe that~$T$ separates~$a^*$ from~$b^*$ and~$c^*$.
    By \Cref{dualtuttecycle} we obtain that~$G$ has a~$2$-coloring such that each color induces a triangle-forest, coloring~$a,b,c$ with the same color.

    If~$\Delta$ is heterochromatic, then we can assume without loss of generality that~$a$ is colored differently from~$b$ and~$c$.
    By \cref{Tutte}, we can take a Tutte path~$T$ from~$B^*$ through~$A^*$ to~$C^*$. 
    Now add to~$T$ the path~$B^*,\Delta^*,C^*$ obtaining a cycle~$T'$.
    Observe that~$T'$ separates~$a^*$ from~$b^*$ and~$c^*$.
    Further, since~$T'$ contains~$A^*$, the face~$f$ incident to~$b^*$ and~$c^*$ is also separated from~$b^*$ and~$c^*$.
    Since~$T$ was a Tutte path in~$H$ and the only new vertex~$\Delta^*$ is on~$T'$ and has all its neighbors in~$T$, we have that~$T'$ is a Tutte cycle of~$G^*$.
    Together with~\Cref{dualtuttecycle}, we obtain that~$G$ has a~$2$-coloring such that each color induces a triangle-forest, coloring~$a$ different from~$b,c$ without a monochromatic triangle containing the edge~$bc$.
 \end{proof}

\begin{thm}\label{thm:bipartition}
    The vertices of any planar graph~$G$ can be~$2$-colored such that each color class induces a triangle-forest.
    Moreover, there is such a coloring for any prescribed precoloring of any fixed triangle~$\Delta$.
\end{thm}
\begin{proof}
    Add edges or vertices to~$G$ in order to turn it into a triangulation. 
    Removing these elements from the end result, still gives a vertex-partition into two triangle-forest.

    Let~$\Delta$ be the fixed triangle.
    We proceed by induction on the number of vertices. 
    If~$G$ is~$4$-connected, then~$\Delta$ is a face and the result follows immediately from~\Cref{4connected}.
    Otherwise, if~$\Delta$ is separating, let us pick one separating triangle~$\Delta'$.
    If~$\Delta'=\Delta$ then apply induction to the interior and exterior of~$\Delta$ with respect to the prescribed coloring on~$\Delta$.
    If~$\Delta$ is (without loss of generality) on the exterior of~$\Delta'$, then remove the interior of~$\Delta'$ and apply induction resulting in some coloring on~$\Delta'$.
    Now apply induction with respect to this precoloring on~$\Delta'$ to the interior of~$\Delta'$.
\end{proof}

\subsection{Tightness and possible strengthenings}
\Cref{thm:bipartition} implies that every planar graph~$G$ on~$n$ vertices contains an induced triangle-forest on at least~$n/2$ vertices.
On the other hand, there are planar graphs where every induced triangle-forest contains at most half the vertices.
Observe for example that any induced triangle-forest in the octahedron graph contains at most~$4$ of its~$8$ vertices.
Thus, any vertex-disjoint union of octahedra (also with any set of additional edges, e.g., to obtain a triangulation) has no induced triangle-forest on more than half of its vertices.

\Cref{thm:bipartition} cannot be strengthened to vertex-partitioning every planar graph into one forest and one triangle-forest.
To see this, take~$G$ to be the dual graph of a cyclically~$4$-edge-connected~$3$-connected planar cubic non-Hamiltonian graph.
(Such graphs exist from~$42$ vertices on, see~\cite{ABHM00}.)
Thus,~$G$ is a~$4$-connected planar triangulation that cannot be vertex-partitioned into two forests.
Now, stack a triangle~$T$ into each face~$F$ of~$G$, such that~$T\cup F$ induces an octahedron.
Suppose that the obtained graph~$G'$ has a vertex-partition into one forest and one triangle-forest.
Then, some triangular face~$F$ of~$G$ must be in the triangle-forest.
But then the triangle~$T$ of~$G'$ stacked into~$F$ must be entirely part of the forest --- contradiction. 

\begin{quest}
Can every planar graph be vertex-partitioned into one forest and one chordal graph, or into one forest and one outerplanar graph?
\end{quest}


\section{Graphs on surfaces}

We will discuss possible extensions to surfaces of higher genus, see~\cite{MC01} for undefined notions. Indeed, \Cref{thm:bipartition} does not extend to graphs embeddable in other surfaces. 
It does not hold on the torus, since the~$K_7$ embeds on this surfaces but cannot be vertex-partitioned into two triangle-forests.
Also for the projective plane there are graphs that cannot be vertex-partitioned into two triangle-forests, as for example the graph in~\Cref{projective}.

\begin{figure}[htp]
    \centering
    \includegraphics[width=.5\textwidth]{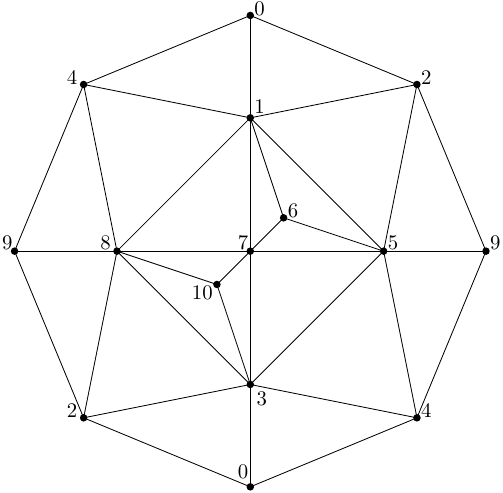}
    \caption{A projective planar graph (with g6-code \texttt{J|tyIlxJGb?}) that cannot be vertex-partitioned into two triangle-forests.}\label{projective}
\end{figure}

Next, we show that~\Cref{thm:bipartition} cannot even be extended to locally planar graphs. To do so, we will construct graphs embeddable on a surface $\Sigma$ such that in every $2$-coloring of their vertices, there is a long monochromatic cycle.
This will easily follows from the following lemma, which is inspired by an answer on mathoverflow~\cite{mathoverflow}.

\begin{lem}\label{lem:hex_lemma_surface}
    Let $\Sigma$ be a surface non-isomorphic to the sphere, and let $G$ be a graph cellularly embedded in $\Sigma$.
    Let $\phi$ be a $2$-coloring of $G$ such that every face $f$ of $G$ which is not a triangle is monochromatic.
    Then there exists a monochromatic non-contractible cycle $C$ in $G$.
\end{lem}

\begin{proof}
    Suppose for contradiction that $\Sigma,G$ and $\phi\colon V(G) \to \{1,2\}$ are a counter-example with $|V(G)|$ minimum.
    If $\phi^{-1}(i) = \emptyset$ for some $i \in \{1,2\}$, then $\phi$ is constant $3-i$.
    Since $G$ is cellularly embedded in $\Sigma$ and $\Sigma$ is not the sphere, $G$ contains a non-contractible cycle $C$, which is then monochromatic.
    Now suppose that $\phi^{-1}(1), \phi^{-1}(2) \neq \emptyset$.
    
    Let $K$ be a connected component of $G[\phi^{-1}(1)]$.
    Let $f$ be a face of $K$ whose interior contains at least one vertex.
    Such a face exists since $\phi^{-1}(2) \neq \emptyset$.
    Let $\Int(f)$ be the (possibly empty) embedded graph induced by the vertices of $G$ lying in the interior of $f$.
    We denote by $\Out(f)$ the face of $\Int(f)$ containing $V(K)$.
    
    We claim that every face of $\Int(f)$ is either a triangle or monochromatic for $\phi\vert_{V(\Int(f))}$.
    Indeed, if $f'$ is a face of $\Int(f)$, then either $f'$ is a face of $G$ and so is either a triangle or monochromatic,
    or $f' = \Out(f)$.
    In the latter case, if $f'$ is neither a triangle nor monochromatic, then there are two consecutive vertices $u,v$ along $f'$ in $\Int(f)$
    with $\phi(u)=1$ and $\phi(v)=2$.
    By construction, these two vertices belongs to a face $f''$ of $G$ that contains a vertex in $V(K)$.
    Since $f''$ is a face of $G$ which is not monochromatic, $f'$ is a triangle.
    In particular, there is an edge colored $1$ between $V(K)$ and $V(\Int(f))$, contradicting the fact that $K$ is a connected component
    of $G[\phi^{-1}(1)]$.
    This proves that every face of $\Int(f)$ is either a triangle or monochromatic.

    With the same argument, one can show that every face of $G-V(\Int(f))$ is either a triangle or monochromatic.

    If $\Int(f)$ is cellularly embedded in $\Sigma$, then by minimality of $|V(K)|$,
    it contains a monochromatic non-contractible cycle $C$ and we are done.
    Otherwise, $G-V(\Int(f))$ is cellularly embedded in $\Sigma$, and so by minimality of $|V(G)|$, $G - V(\Int(f))$ contains a non-contractible monochromatic cycle $C$.
    This proves the lemma.
\end{proof}

\begin{cor}\label{cor:not2}
    Let $\Sigma$ be a surface non-isomorphic to the sphere. For every positive integer $\ell$, there is a graph $G$ embeddable in $\Sigma$
    such that for every $2$-coloring of $V(G)$, there is a monochromatic cycle of length at least $\ell$ in $G$.
    In particular, for $\ell \geq 4$, $G$ does not admit a partition of $V(G)$ into two induced triangle-forests.
\end{cor}

\begin{proof}
    Let $\ell$ be  positive integer.
    Let $G$ be a triangulation of $\Sigma$ such that every non-contractible cycle of $G$ has length at least $\ell$.
    Then, by Lemma~\ref{lem:hex_lemma_surface}, for every $2$-coloring of $G$, $G$ contains a monochromatic non-contractible cycle $C$, which must have length at least $\ell$.
\end{proof}

On the other hand, we show that any graph embedded in a fixed surface $\Sigma$ with no small non-contractible cycle can be partitioned into four induced forest.
This is a consequence of the following theorem.
We say that a graph $G$ is \emph{acyclically $k$-colorable} for a positive integer $k$,
if $G$ admits a proper $k$-coloring of its vertices, such that for every pair $i,j$ of colors, the union of color class of $i$ and color class of $j$
induces a forest in $G$.

\begin{thm}[Kawarabayashi and Mohar~\cite{Kawarabayashi2010}]{\ \\}
    Let $\Sigma$ be a surface. There is an integer $\ell$ such that for every graph $G$ embedded in $\Sigma$,
    if $G$ has no non-contractible cycle of length at most $\ell$,
    then $G$ is acyclically $7$-colorable.
\end{thm}

\begin{cor}\label{cor:4}
    Let $\Sigma$ be a surface. There is an integer $\ell$ such that for every graph $G$ embedded in $\Sigma$,
    if $G$ has no non-contractible cycle of length at most $\ell$,
    then $G$ can be partitioned into four induced forests.
\end{cor}

Note that a positive answer to the following would be a weakening of~\cite[Conjecture 1.3]{Kawarabayashi2010}:

\begin{quest}
Can every graph embedded in $\Sigma$ with no small non-contractible cycle be partitioned into three (triangle-)forests?
\end{quest}

\subsubsection*{Acknowledgements.}
We thank Meike Hatzel, Bobby Miraftab, and Sandra Kiefer for fruitful discussions and the organizers and participants of the Tenth and the Eleventh Annual Workshop on Geometry and Graphs, 2023 and 2024 at Bellairs Research Institute for a great atmosphere.

KK is supported by the Spanish Research Agency through grants RYC-2017-22701, PID2019-104844GB-I00, PID2022-137283NB-C22 and the Severo Ochoa and María de Maeztu Program for Centers and Units of Excellence in R\&D (CEX2020-001084-M) and by the French Research Agency through ANR project DAGDigDec: ANR-21-CE48-0012.

TU is funded by the Deutsche Forschungsgemeinschaft (DFG, German Research Foundation) - 520723789.

\bibliography{lit}

\begin{thebibliography}{10}

\bibitem{mathoverflow}
mathoverflow post {``Study of Hex on the Torus''}, 2017.
\newblock last access June 2024.
\newblock URL:
  \url{https://mathoverflow.net/questions/282088/study-of-hex-on-the-torus}.

\bibitem{ABHM00}
R.~E.~L. Aldred, S.~Bau, D.~A. Holton, and Brendan~D. McKay.
\newblock Nonhamiltonian 3-connected cubic planar graphs.
\newblock {\em SIAM J. Discrete Math.}, 13(1):25--32, 2000.
\newblock \href {https://doi.org/10.1137/S0895480198348665}
  {\path{doi:10.1137/S0895480198348665}}.

\bibitem{AEFG16}
Patrizio Angelini, William Evans, Fabrizio Frati, and Joachim Gudmundsson.
\newblock {SEFE} without mapping via large induced outerplane graphs in plane
  graphs.
\newblock {\em J. Graph Theory}, 82(1):45--64, 2016.
\newblock \href {https://doi.org/10.1002/jgt.21884}
  {\path{doi:10.1002/jgt.21884}}.

\bibitem{CK69}
Gary Chartrand and Hudson~V. Kronk.
\newblock The point-arboricity of planar graphs.
\newblock {\em J. Lond. Math. Soc.}, 44:612--616, 1969.
\newblock \href {https://doi.org/10.1112/jlms/s1-44.1.612}
  {\path{doi:10.1112/jlms/s1-44.1.612}}.

\bibitem{HM88}
D.~A. Holton and B.~D. McKay.
\newblock The smallest non-{Hamiltonian} 3-connected cubic planar graphs have
  38 vertices.
\newblock {\em J. Comb. Theory, Ser. B}, 45(3):305--319, 1988.
\newblock \href {https://doi.org/10.1016/0095-8956(88)90075-5}
  {\path{doi:10.1016/0095-8956(88)90075-5}}.

\bibitem{Kawarabayashi2010}
Ken-ichi Kawarabayashi and Bojan Mohar.
\newblock Star coloring and acyclic coloring of locally planar graphs.
\newblock {\em SIAM J. Discrete Math.}, 24(1):56–71, January 2010.
\newblock \href {https://doi.org/10.1137/060674211}
  {\path{doi:10.1137/060674211}}.

\bibitem{LM17}
Zhentao Li and Bojan Mohar.
\newblock Planar digraphs of digirth four are 2-colorable.
\newblock {\em SIAM J. Discrete Math.}, 31(3):2201--2205, 2017.
\newblock \href {https://doi.org/10.1137/16M108080X}
  {\path{doi:10.1137/16M108080X}}.

\bibitem{MC01}
Bojan Mohar and Carsten Thomassen.
\newblock {\em Graphs on surfaces}.
\newblock Baltimore, MD: Johns Hopkins University Press, 2001.

\bibitem{RW08}
Andr{\'e} Raspaud and Weifan Wang.
\newblock On the vertex-arboricity of planar graphs.
\newblock {\em Eur. J. Comb.}, 29(4):1064--1075, 2008.
\newblock \href {https://doi.org/10.1016/j.ejc.2007.11.022}
  {\path{doi:10.1016/j.ejc.2007.11.022}}.

\bibitem{San96}
Daniel~P. Sanders.
\newblock On {Hamilton} cycles in certain planar graphs.
\newblock {\em J. Graph Theory}, 21(1):43--50, 1996.
\newblock \href
  {https://doi.org/10.1002/(SICI)1097-0118(199601)21:1<43::AID-JGT6>3.0.CO;2-M}
  {\path{doi:10.1002/(SICI)1097-0118(199601)21:1<43::AID-JGT6>3.0.CO;2-M}}.

\bibitem{Tai84}
Peter~Guthrie Tait.
\newblock Listing`s topology.
\newblock {\em Phil. Mag. (5)}, 17:30--46, 1884.

\bibitem{TY94}
Robin Thomas and Xingxing Yu.
\newblock 4-connected projective planar graphs are {Hamiltonian}.
\newblock {\em J. Comb. Theory, Ser. B}, 62(1):114--132, 1994.
\newblock \href {https://doi.org/10.1006/jctb.1994.1058}
  {\path{doi:10.1006/jctb.1994.1058}}.

\bibitem{Tho95}
Carsten Thomassen.
\newblock Decomposing a planar graph into an independent set and a 3-degenerate
  graph.
\newblock {\em J. Comb. Theory, Ser. B}, 83(2):262--271, 2001.
\newblock \href {https://doi.org/10.1006/jctb.2001.2056}
  {\path{doi:10.1006/jctb.2001.2056}}.

\bibitem{Tut46}
William~T. Tutte.
\newblock On {Hamiltonian} circuits.
\newblock {\em J. Lond. Math. Soc.}, 21:98--101, 1946.
\newblock \href {https://doi.org/10.1112/jlms/s1-21.2.98}
  {\path{doi:10.1112/jlms/s1-21.2.98}}.

\bibitem{Tut56}
William~T. Tutte.
\newblock A theorem on planar graphs.
\newblock {\em Trans. Am. Math. Soc.}, 82:99--116, 1956.
\newblock \href {https://doi.org/10.2307/1992980} {\path{doi:10.2307/1992980}}.

\bibitem{Wu10}
Franklin Wu.
\newblock {\em Induced Forests in Planar Graphs}.
\newblock UCSD (Master's Thesis), 2010.
\newblock URL:
  \url{https://www.math.ucsd.edu/sites/math.ucsd.edu/files/undergrad/honors-program/honors-theses/2009-2010/Franklin_Wu_Honors_Thesis.pdf}.

\end{thebibliography}
\bibliographystyle{plainurl}

\end{document}